\documentclass[11pt]{article}
\usepackage{enumerate}
\usepackage{amssymb,a4wide,latexsym,makeidx,epsfig,fleqn}
\usepackage{amsthm}
\usepackage{amsmath}
\usepackage{enumerate}
\usepackage{color}
\usepackage[numbers,sort&compress]{natbib}
\newtheorem{theorem}{Theorem}[section]
\newtheorem{lemma}{Lemma}[section]

\newtheorem{case}{Case}[section]
\newtheorem{subcase}{Subcase}[case]

\newtheorem{claim}{Claim}[section]

\begin{document}
\textwidth 150mm \textheight 225mm
\title{The generalized Tur\'{a}n number of spanning linear forests \thanks{Supported by the National Natural Science Foundation of
China (No. 11871398) and the Seed Foundation of Innovation and Creation for Graduate Students in Northwestern Polytechnical University (No. CX2020190).}}
\author{{Lin-Peng Zhang\textsuperscript{a,b}, Ligong Wang\textsuperscript{a,b,}\footnote{Corresponding author.} and Jiale Zhou\textsuperscript{a}}\\
{\small \textsuperscript{a} School of Mathematics and Statistics}\\
{\small Northwestern Polytechnical University, Xi'an, Shaanxi 710129, P.R. China.}\\
{\small \textsuperscript{b} Xi'an-Budapest Joint Research Center for Combinatorics}\\
{\small  Northwestern Polytechnical University, Xi'an, Shaanxi 710129, P.R. China.}\\
{\small E-mail: lpzhangmath@163.com, lgwangmath@163.com, zjl0508math@mail.nwpu.edu.cn}}
\date{}
\maketitle
\begin{center}
\begin{minipage}{135mm}
\vskip 0.3cm
\begin{center}
{\small {\bf Abstract}}
\end{center}
{\small Let $\mathcal{F}$ be a family of graphs. A graph $G$ is called \textit{$\mathcal{F}$-free} if for any $F\in \mathcal{F}$, there is no subgraph of $G$ isomorphic to $F$.
Given a graph $T$ and a family of graphs $\mathcal{F}$, the generalized Tur\'{a}n number of $\mathcal{F}$ is the maximum number of copies of $T$ in an $\mathcal{F}$-free graph
on $n$ vertices, denoted by $ex(n,T,\mathcal{F})$. A \textit{linear forest} is a graph whose connected components are all paths or isolated vertices.
Let $\mathcal{L}_{n,k}$ be the family of all linear forests of order $n$ with $k$ edges and $K^*_{s,t}$ a graph obtained from $K_{s,t}$ by substituting the part of size $s$ with a clique of the same size.
In this paper, we determine the exact values of $ex(n,K_s,\mathcal{L}_{n,k})$
and $ex(n,K^*_{s,t},\mathcal{L}_{n,k})$.
Also, we study the case of this problem when the \textit{``host graph''} is bipartite. Denote by $ex_{bip}(n,T,\mathcal{F})$ the maximum possible number of copies of $T$ in an $\mathcal{F}$-free
bipartite graph with each part of size $n$. We determine the exact value of $ex_{bip}(n,K_{s,t},\mathcal{L}_{n,k})$.
Our proof is mainly based on the shifting method.
}
\end{minipage}
\end{center}

\section{Introduction}
\quad\quad Let $\mathcal{F}$ be a family of graphs. A graph $G$ is called \textit{$\mathcal{F}$-free} if for any $F\in \mathcal{F}$, there is no subgraph of $G$ isomorphic to $F$.
Given a graph $T$ and a family of graphs $\mathcal{F}$, the \textit{generalized Tur\'{a}n number} of $\mathcal{F}$ is the maximum number of copies of $T$ in an $\mathcal{F}$-free graph
on $n$ vertices, denoted by $ex(n,T,\mathcal{F})$. When $T=K_2$, it reduces to the classical \textit{Tur\'{a}n number} $ex(n,\mathcal{F})$. When $\mathcal{F}$ contains only one simple graph $F$, we write
$ex(n,T,F)$ instead of $ex(n,T,\{F\})$. In \cite{zy}, Zykov determined the exact value of $ex(n,K_s,K_t)$. Let $P_k$ be the path on $k$ vertices and $\mathcal{C}_{\ge k}$ the family of all cycles with length at least $k$.
In \cite{Luo}, Luo determined the upper bounds on $ex(n,K_s,P_k)$ and $ex(n,K_s,\mathcal{C}_{\ge k})$. The two results generalized the Erd\H{o}s-Gallai's Theorem on paths and cycles \cite{E-G}. Recently, the problem
to estimate generalized Tur\'{a}n number has received a lot of attention, refer to \cite{Alon,G1,G2,G3,G4,G5,G6,G7,M-Q,N-P}.

A \textit{matching} in a graph $G$ is a subset of the edge set of $G$ consisting of pairwise disjoint edges. Denote by $M_{k}$ a matching containing $k$ edges.
For an integer $s$, we denote by $K_s$ and $E_s$ the complete graph on $s$ vertices
and the empty graph on $s$ vertices, respectively. The \textit{join} of two disjoint graphs $H_1$ and $H_2$, denoted by $H_1\vee H_2$, is the graph whose vertex set is $V(H_1\vee H_2)=V(H_1)\cup V(H_2)$
and edge set is $E(H_1\vee H_2)=E(H_1)\cup E(H_2)\cup \{xy:x\in V(H_1), y\in V(H_2)\}$. In \cite{E-G}, Erd\H{o}s and Gallai determined the exact value of $ex(n, M_{k+1})$. For the lower bound, the constructions $K_{2k+1}$ and
$K_k\vee E_{n-k}$ are $M_{k+1}$-free graphs with the required number of edges.
\begin{theorem}[\cite{E-G}]\label{ch1:E-G}
For any $n\ge 2k+1$, we have
$$
ex(n,M_{k+1})=max\left\{\dbinom{2k+1}{2}, \dbinom{k}{2}+k(n-k)\right\}.
$$
\end{theorem}

In \cite{wang}, Wang determined the exact value of $ex(n,K_s,M_{k+1})$ by using the \textit{shifting method}, which generalized Theorem \ref{ch1:E-G}.
\begin{theorem}[\cite{wang}]\label{ch1:th1}
For any $s\ge 2$ and $n\ge 2k+1$, we have
$$
ex(n,K_s,M_{k+1})=max\left\{\dbinom{2k+1}{s},\dbinom{k}{s}+(n-k)\dbinom{k}{s-1}\right\}.
$$
\end{theorem}
Let $K^*_{s,t}$ be a graph obtained from $K_{s,t}$ by substituting the part of size $s$
with a clique of the same size. Wang \cite{wang} also determined the exact value of $ex(n,K^*_{s,t}, M_{k+1})$.
\begin{theorem}[\cite{wang}]\label{ch1:th2}
For any $s\ge 1$, $t\ge 2$ and $n\ge 2k+1$, we have
\begin{align}
\nonumber ex(n,K^*_{s,t}, M_{k+1})=&max\left\{{\dbinom{2k+1}{s+t}\dbinom{s+t}{t},\dbinom{k}{s}\dbinom{n-s}{t}}\right.\\
\nonumber &\left.{+(n-k)\dbinom{k}{s+t-1}\dbinom{s+t-1}{t}}\right\}.
\end{align}
\end{theorem}
For the lower bounds of $ex(n,K_s,M_{k+1})$ and $ex(n,K^*_{s,t},M_{k+1})$,
the constructions $K_{2k+1}$ and $K_k\vee E_{n-k}$ are $M_{k+1}$-free graphs with the required number of $s$-cliques and $K^*_{s,t}$.

A matching can also be viewed as a forest whose components are all paths with length one. A \textit{linear forest} is a graph whose connected components are all paths or isolated vertices.
Denote by $\mathcal{L}_{n,k}$ the family of all linear forests of order $n$ with $k$ edges. Recently, Ning and Wang \cite{N-W} determined the exact value of $ex(n, \mathcal{L}_{n,k})$.
\begin{theorem}[\cite{N-W}]\label{ch1:th4}
For any $1\le k\le n-1$,
$$
ex(n,\mathcal{L}_{n,k})=max\left\{\dbinom{k}{2},\dbinom{n}{2}-\dbinom{n-\left\lfloor\frac{k-1}{2}\right\rfloor}{2}+c\right\},
$$
where $c=0$ if $k$ is odd, and $c=1$ otherwise.
\end{theorem}
For the lower bound of Theorem \ref{ch1:th4}, the constructions $K_k$ and $K_{\frac{k-1}{2}}\vee E_{n-\frac{k-1}{2}}$ are $\mathcal{L}_{n,k}$-free graphs with the required number of edges for an odd $k$,
and the constructions $K_k$ and $K_{\frac{k}{2}-1}\vee (E_{n-\frac{k}{2}-1}\cup K_2)$ are $\mathcal{L}_{n,k}$-free graphs with the required number of edges for an even $k$.
They mainly used \textit{the closure operation} in the proof. They call this approach the \textit{closure technique for Tur\'{a}n problems}.

Motivated by the Tur\'{a}n number of $\mathcal{L}_{n,k}$ and the generalized Tur\'{a}n number of matchings, in this paper, we consider the generalized Tur\'{a}n number of $\mathcal{L}_{n,k}$ and determine the exact
values of $ex(n, K_s, \mathcal{L}_{n,k})$ and $ex(n, K^*_{s,t}, \mathcal{L}_{n,k})$.
\begin{theorem}\label{ch1:th5}
For any $s\ge 2$ and $n\ge k+1$,
$$
ex(n,K_s,\mathcal{L}_{n,k})=
max\left\{\dbinom{k}{s}, \dbinom{\left\lceil\frac{k+1}{2}\right\rceil}{s}+
\bigg(n-\left\lceil\frac{k+1}{2}\right\rceil\bigg)\dbinom{\left\lfloor\frac{k-1}{2}\right\rfloor}{s-1}\right\}.
$$
\end{theorem}

\begin{theorem}\label{ch1:th6}
For any $s\ge 1$, $t\ge 2$ and $n\ge k+1$,
\begin{align}
\nonumber ex(n,K^*_{s,t},\mathcal{L}_{n,k})&=max\left\{{\dbinom{k}{s+t}\dbinom{s+t}{t}, \dbinom{\left\lfloor\frac{k-1}{2}\right\rfloor}{s}\dbinom{n-s}{t}}\right. \\
\nonumber &+\left.{\bigg(n-\left\lceil\frac{k+1}{2}\right\rceil\bigg)\dbinom{\left\lfloor\frac{k-1}{2}\right\rfloor}{s-1}\dbinom{\left\lfloor\frac{k-1}{2}\right\rfloor-s+1}{t}}\right.\\
\nonumber &\left.{+\bigg(\dbinom{\left\lceil\frac{k+1}{2}\right\rceil}{s}-\dbinom{
\left\lfloor\frac{k-1}{2}\right\rfloor}{s}\bigg)\dbinom{\left\lceil\frac{k+1}{2}\right\rceil-s}{t}}
\right\}.
\end{align}
\end{theorem}

In \cite{wang}, Wang also studied the bipartite case of the problem.
Denote by $ex_{bip}(n,T,F)$ the maximum possible number of copies of $T$ in a bipartite $F$-free graph with each part of equal size $n$. Wang \cite{wang} determined the exact value of $ex_{bip}(n, T, M_{k+1})$ for $T=K_{s,t}$.
\begin{theorem}[\cite{wang}]\label{ch1:th3}
For any $s,t\ge 2$ and $n\ge k$,
$$
ex_{bip}(n, K_{s,t}, M_{k+1})=
\begin{cases}
\dbinom{k}{s}\dbinom{n}{t}+\dbinom{k}{t}\dbinom{n}{s}, & s\neq t, \\
\dbinom{k}{s}\dbinom{n}{s}, & s=t.
\end{cases}
$$
\end{theorem}

Also, we further generalize their result for $\mathcal{L}_{n,k}$. In particular, we prove the following theorem.
\begin{theorem}\label{ch1:th7}
Let $s,t,n$ be three positive integers, $s,t\ge 1$ and $n\ge \lceil\frac{k-1}{2}\rceil$.
If $k$ is odd, then
$$
ex_{bip}(n,K_{s,t},\mathcal{L}_{n,k}) =
\begin{cases}
\dbinom{\frac{k-1}{2}}{s}\dbinom{n}{s},  & s=t, \\
\dbinom{\frac{k-1}{2}}{s}\dbinom{n}{t}+\dbinom{\frac{k-1}{2}}{t}\dbinom{n}{s}, & s\neq t.
\end{cases}
$$
If $k$ is even, then for $s=t$,
$$
ex_{bip}(n,K_{s,t},\mathcal{L}_{n,k}) =
\begin{cases}
\cfrac{kn}{2}-\cfrac{k}{2}+1,  & s=1, \\
\dbinom{\frac{k}{2}-1}{s}\dbinom{n}{s}, & s\ge 2,
\end{cases}
$$
for $s\neq t$,
$$
ex_{bip}(n,K_{s,t},\mathcal{L}_{n,k}) =
\begin{cases}
\cfrac{k}{2}\dbinom{n}{t}+(n-1)\dbinom{\frac{k}{2}-1}{t},  & s=1, t\ge 2, \\
\cfrac{k}{2}\dbinom{n}{s}+(n-1)\dbinom{\frac{k}{2}-1}{s}, & s\ge 2, t=1,\\
\dbinom{\frac{k}{2}-1}{t}\dbinom{n}{s}+\dbinom{\frac{k}{2}-1}{s}\dbinom{n}{t} & s,t\ge 2.
\end{cases}
$$
\end{theorem}

The paper is organised as follows. In Section \ref{se2}, we introduce the shifting operation on graphs and some properties of this operation. In Section \ref{se3}, we prove Theorems
\ref{ch1:th5} and \ref{ch1:th6}. In Section \ref{se4}, we prove Theorem \ref{ch1:th7}.

\section{Preliminaries} \label{se2}
\qquad In this section we will present some notations needed in the subsequent sections, and then introduce the shifting operation on graphs and some properties of this operation.

Denote by $[n]$ the set $\{1,2,\cdots,n\}$. Let $G$ be a simple graph, we denote by $V(G)$ and $E(G)$ the vertex set and the edge set of $G$, respectively.
Denote by $e(G)$ the number of edges of $G$. For a graph $G$ and its subgraph $H$, we use $G-H$ to denote a graph obtained from G by deleting all vertices of $H$ with all incident edges. For any subset $S\subset V(G)$, we denote by $e(S)$ the number of edges with two endpoints in $S$ and
$G[S]$ the subgraph induced by $S$. For two disjoint subsets $S,T\subset V(G)$, we denote by $G[S,T]$ the induced bipartite graph between $S$ and $T$.
Let $\bar{S}=V(G)\setminus S$. Let $e(S,\bar{S})$ be the number of edges with one endpoint in $S$ and the other endpoint in $\bar{S}$. For any $v\in V(G)$ and $S\subset V(G)$, we denote by $d_S(v)$
the number of neighbors of $v$ in $S$. Denote by $\nu(G)$ the number of edges in a maximum matching of $G$. Let $\mathcal{N}(G,T)$ be the number of $T$ in $G$.

Suppose a graph $G$ has vertex set $V(G)=[n]$ and edge set $E(G)$. Here, edges in $E(G)$ are viewed as subsets of $V(G)$.
For $1\le i<j\le n$ and $e\in E(G)$, we define a \textit{shifting operation} (also known as \textit{Kelmans transformation} \cite{Kelmans}) $S_{ij}$ on $e$
as follows:
$$
S_{ij}(e)=
\begin{cases}
(e-\{j\})\cup \{i\},  & \mbox{if }j\in e, i\notin e\mbox{ and } (e-\{j\})\cup \{i\}\notin E(G), \\
e, & \mbox{otherwise. }
\end{cases}
$$
We define $S_{ij}(G)$ to be a graph on vertex set $V(G)$ with edge set $\{S_{ij}(e):e\in E(G)\}$.

It is not hard to see that $e(S_{ij}(G))=e(G)$. Further, we show the following lemma.
\begin{lemma}\label{ch2:lem1}
Suppose $G$ is a graph on vertex set $[n]$. Then for any $1\le i<j\le n$, we have that $S_{ij}(G)$ is also $\mathcal{L}_{n,k}$-free if $G$ is
$\mathcal{L}_{n,k}$-free.
\end{lemma}
\begin{proof}
Assume that $G$ is $\mathcal{L}_{n,k}$-free. Suppose to the contrary that $S_{ij}(G)$ is not $\mathcal{L}_{n,k}$-free, then $S_{ij}(G)$ contains a graph of $\mathcal{L}_{n,k}$ as its subgraph. In the following, we will discuss three cases.
\begin{case}
$S_{ij}(G)$ contains a copy of $P_{k+1}$ as its subgraph.
\end{case}
Generally, we denote that $P_{k+1}=v_0v_1v_2\cdots v_k$.
\begin{subcase}
Assume that $i=v_s$, $j=v_t$ where $|s-t|\ge 3$, $0<s,t<k$.
\end{subcase}
Without loss of generality, we assume that $s<t$. We suppose that none of the edges $\{v_{s-1},v_s\}$ and $\{v_s,v_{s+1}\}$ are edges of the original graph $G$. However, by the definition of the transformation $\{v_t, v_{s-1}\}$, $\{v_t, v_{s+1}\}$, $\{v_s, v_{t-1}\}$ and $\{v_s, v_{t+1}\}$ are all edges of the original graph $G$ and so $v_0\cdots v_{s-1}v_tv_{s+1}v_{s+2}\cdots v_{t-1}v_sv_{t+1}v_{t+2}\cdots v_k$ is a path of the same length in the original graph $G$ using the same vertices, a contradiction. If both of the edges $\{v_{s-1},v_s\}$ and $\{v_s,v_{s+1}\}$ are edges of the original graph $G$, then the path $P_{k+1}=v_0v_1v_2\cdots v_k$ is also a subgraph of the original graph $G$, a contradiction. If $\{v_{s-1},v_s\}\in E(G)$ but $\{v_s,v_{s+1}\}\notin E(G)$, then $v_0\cdots v_{s-1}v_sv_{t-1}v_{t-2}\cdots v_{s+1}v_tv_{t+1}\cdots v_k$ is a path of the same length in the original graph $G$ using the same vertices, a contradiction. Similarly, if $\{v_s,v_{s+1}\}\in E(G)$ but $\{v_{s-1},v_s\}\notin E(G)$, then $v_0\cdots v_{s-1}v_tv_{t-1}\cdots v_{s+1}v_sv_{t+1}\cdots v_k$ is a path of the same length in the original graph $G$ using the same vertices, a contradiction.
\begin{subcase}
Assume that $i=v_s$, $j=v_t$ where $|s-t|\le 2$, $0<s,t<k$.
\end{subcase}
Without loss of generality, we assume that $s<t$. If $t-s=2$, then by the definition of the transformation the edges $\{v_s,v_{s+1}\}$ and $\{v_{s+1},v_t\}$ are edges of the original graph $G$. If the edge $\{v_{s-1},v_s\}$ is the edge of the original graph $G$, then the path $P_{k+1}=v_0v_1v_2\cdots v_k$ is also a subgraph of the original graph $G$, a contradiction. If the edge $\{v_{s-1},v_s\}$ is not the edge of the original graph $G$, then $v_0\cdots v_{s-1}v_tv_{t-1}v_sv_{t+1}\cdots v_k$ is a path of the same length in the original graph $G$ using the same vertices, a contradiction. If $t-s=1$, then by the definition of the transformation the edge $\{v_s,v_t\}$ is the edge of the original graph $G$. If the edge $\{v_{s-1},v_s\}$ is the edge of the original graph $G$, then the path $P_{k+1}=v_0v_1v_2\cdots v_k$ is also a subgraph of the original graph $G$, a contradiction. If the edge $\{v_{s-1},v_s\}$ is not the edge of the original graph $G$, then $v_0\cdots v_{s-1}v_tv_sv_{t+1}\cdots v_k$ is a path of the same length in the original graph $G$ using the same vertices, a contradiction.
\begin{subcase}
Assume that $i=v_s$, $j=v_t$ where either one of $0$ and $k$ is in $\{s,t\}$ or $\{0,k\}=\{s,t\}$.
\end{subcase}
Without loss of generality, we assume that $s<t$. Assume that $s=0$ and $t-s\le 2$. It follows from the definition of the transformation that the path $P_{k+1}=v_0v_1v_2\cdots v_k$ is also a subgraph of the original graph $G$, a contradiction. Assume that $s=0$ and $t-s\ge 3$. If $\{v_0,v_1\}\in E(G)$, then the path $P_{k+1}=v_0v_1v_2\cdots v_k$ is also a subgraph of the original graph $G$, a contradiction. If $\{v_0,v_1\}\notin E(G)$, then $v_0v_{t-1}\cdots v_1v_tv_{t+1}\cdots v_k$ is a path of the same length in the original graph $G$ using the same vertices, a contradiction. Similarly, if $t=k$, $t-s\le 2$ and $\{v_{s-1},v_s\}\in E(G)$, then the path $P_{k+1}=v_0v_1v_2\cdots v_k$ is also a subgraph of the original graph $G$, a contradiction. If $t=k$, $t-s\le 2$ and $\{v_{s-1},v_s\}\notin E(G)$, then by the definition of the transformation $v_0v_1\cdots v_{s-1}v_k\cdots v_s$ is a path of the same length in the original graph $G$ using the same vertices, a contradiction. If $t=k$, $t-s\ge 3$ and none of the edges $\{v_{s-1},v_s\}$ and $\{v_s,v_{s+1}\}$ are edges of the original graph $G$, then $v_0\cdots v_{s-1}v_kv_{s+1}v_{s+2}\cdots v_{k-1}v_s$ is a path of the same length in the original graph $G$ using the same vertices, a contradiction. If both of the edges $\{v_{s-1},v_s\}$ and $\{v_s,v_{s+1}\}$ are edges of the original graph $G$, then the path $P_{k+1}=v_0v_1v_2\cdots v_k$ is also a subgraph of the original graph $G$, a contradiction. If $\{v_{s-1},v_s\}\in E(G)$ but $\{v_s,v_{s+1}\}\notin E(G)$, then $v_0\cdots v_{s-1}v_sv_{k-1}v_{k-2}\cdots v_{s+1}v_k$ is a path of the same length in the original graph $G$ using the same vertices, a contradiction. Similarly, if $\{v_s,v_{s+1}\}\in E(G)$ but $\{v_{s-1},v_s\}\notin E(G)$, then $v_0\cdots v_{s-1}v_kv_{k-1}\cdots v_{s+1}v_s$ is a path of the same length in the original graph $G$ using the same vertices, a contradiction.
\begin{subcase}\label{subcase2.1.4}
Assume that $i=v_s$ where $0\le s\le k$ and $j\notin \{v_0,v_1,\cdots, v_k\}$.
\end{subcase}
If $s=0$ and $\{v_0,v_1\}\in E(G)$, then the path $P_{k+1}=v_0v_1v_2\cdots v_k$ is also a subgraph of the original graph $G$, a contradiction. If $s=0$ and $\{v_0,v_1\}\notin E(G)$,then $jv_1v_2\cdots v_k$ is a path of the same length in the original graph $G$, a contradiction. If $s=k$ and $\{v_{k-1},v_k\}\in E(G)$, then the path $P_{k+1}=v_0v_1v_2\cdots v_k$ is also a subgraph of the original graph $G$, a contradiction. If $s=k$ and $\{v_{k-1},v_k\}\notin E(G)$,then $v_0v_1\cdots v_{k-1}j$ is a path of the same length in the original graph $G$, a contradiction.
If $0<s<k$ and none of the edges $\{v_{s-1},v_s\}$ and $\{v_s,v_{s+1}\}$ are edges of the original graph $G$, then $v_0\cdots v_{s-1}v_jv_{s+1}v_{s+2}\cdots v_k$ is a path of the same length in the original graph $G$, a contradiction. If both of the edges $\{v_{s-1},v_s\}$ and $\{v_s,v_{s+1}\}$ are edges of the original graph $G$, then the path $P_{k+1}=v_0v_1v_2\cdots v_k$ is also a subgraph of the original graph $G$, a contradiction. If $\{v_{s-1},v_s\}\in E(G)$ but $\{v_s,v_{s+1}\}\notin E(G)$, then $(v_0v_1\cdots v_{s-1}v_s)\cup (jv_{s+1}v_{s+2}\cdots v_k)\in \mathcal{L}_{n,k}$ is in the original graph $G$, a contradiction. Similarly, if $\{v_s,v_{s+1}\}\in E(G)$ but $\{v_{s-1},v_s\}\notin E(G)$, then $(v_0v_1\cdots v_{s-1}j)\cup (v_sv_{s+1}\cdots v_k)\in \mathcal{L}_{n,k}$ is in the original graph $G$, a contradiction.
\begin{case}
$S_{ij}(G)$ contains a copy of $k\cdot P_2$ as its subgraph.
\end{case}
If $\{i,j\}\in E(k\cdot P_2)$, then the same subgraph $k\cdot P_2$ is in the original graph $G$, a contradiction. If $i,j\in V(k\cdot P_2)$ but
$\{i,j\}\notin E(k\cdot P_2)$, we assume that both of the edges $\{i,i'\}$ and $\{j,j'\}$ are edges of $k\cdot P_2$ in $S_{ij}(G)$. If $\{i,i'\}\in E(G)$, then the same subgraph $k\cdot P_2$ is also in the original graph $G$, a contradiction. If $\{i,i'\}\notin E(G)$, then $(k\cdot P_2-(\{i,i'\}\cup \{j,j'\}))\cup (ij'j)\in \mathcal{L}_{n,k}$ is in the original graph $G$, a contradiction. Assume $i\in V(k\cdot P_2)$ but $j\notin V(k\cdot P_2)$. without loss of generality, we suppose that $\{i,i'\}\in E(k\cdot P_2)$. If $\{i,i'\}\in E(G)$, then $k\cdot P_2$ is also in the original graph $G$, a contradiction. If $\{i,i'\}\notin E(G)$, then $(k\cdot P_2 - \{i,i'\})\cup \{i',j\}\in \mathcal{L}_{n,k}$ is in the original graph $G$, a contradiction.
\begin{case}
$S_{ij}(G)$ contains a copy of graph in $\mathcal{L}_{n,k}\setminus \{P_{k+1},k\cdot P_2\}$ as its subgraph.
\end{case}
Suppose that $S_{ij}(G)$ contains a copy of $H\in \mathcal{L}_{n,k}\setminus \{P_{k+1},k\cdot P_2\}$ as its subgraph. If $i,j\in V(H)$ and $i$ and $j$ is in the
same connected component of $S_{ij}(G)$, then from the above analysis we can find a copy of graph in $\mathcal{L}_{n,k}$ in the original graph $G$, a contradiction. It is easy to see that each component of $H$ is actually a path. Assume $i,j\in V(H)$ and $i$ and $j$ are origin or terminus of two different connected components $P$ and $Q$ in $H$. Suppose that $\{i,i'\}\in E(P)$ and $\{j,j'\}\in E(Q)$. If $\{i,i'\}\in E(G)$, then $H$ is also in $G$, a contradiction. If $\{i,i'\}\notin E(G)$, then $((P-\{i,i'\})\cup \{i',j\})\cup ((Q-\{j,j'\})\cup \{i,j'\})\cup (H-(P\cup Q))\in \mathcal{L}_{n,k}$ is in the original graph $G$, a contradiction. Suppose that $i\in V(P)$ and $j\in V(Q)$, where $P=i_0i_1\cdots i'ii''\cdots i_s$ and $Q=j_0j_1\cdots j'jj''\cdots j_t$ are two connected components of $H$. If none of the edges $\{i',i\}$ and $\{i,i''\}$ are the edges of the original graph $G$, then by the definition of transformation $\{i',j\}$,$\{i'',j\}$,$\{i,j'\}$ and $\{i,j''\}$ are all the edges of the original graph $G$. And so $((P-(i'ii''))\cup (i'ji''))\cup ((Q-(j'jj''))\cup (j'ij''))\cup (H-(P\cup Q))\in \mathcal{L}_{n,k}$ is in $G$, a contradiction. If both of the edges $(i',i)$ and $(i,i'')$ are edges of the original graph $G$, then the graph $H$ is also a subgraph of the original graph $G$, a contradiction. If $\{i',i\}\in E(G)$ but $\{i,i''\}\notin E(G)$, then $(i_0i_1\cdots i'ij'\cdots j_1j_0)\cup (i_si_{s-1}\cdots i''jj''\cdots j_t)\cup (H-(P\cup Q))\in \mathcal{L}_{n,k}$ is in $G$, a contradiction. Similarly, if $\{i,i''\}\in E(G)$ but $\{i',i\}\notin E(G)$, then $(i_0i_1\cdots i'jj'\cdots j_1j_0)\cup (i_s\cdots i''ij''\cdots j_t)\cup (H-(P\cup Q))\in \mathcal{L}_{n,k}$ is in $G$, a contradiction. Suppose that $i,j\in V(H)$ and $j$ is the origin or terminus of a path $Q$ in $H$ but $i\in V(P)$, where $P=i_0i_1\cdots i'ii''\cdots i_s$ and $Q$ are two connected components of $H$. Assume that $jj'j''\in Q$. If none of the edges $\{i',i\}$ and $\{i,i''\}$ are the edges of the original graph $G$, then by the definition of transformation $\{i',j\}$,$\{i'',j\}$ and $\{i,j'\}$ are all the edges of the original graph $G$. And so $((P-(i'ii''))\cup (i'ji''))\cup ((Q-\{j,j'\})\cup \{i,j'\}\cup (H-(P\cup Q))\in \mathcal{L}_{n,k}$ is in $G$, a contradiction. If both of the edges $(i',i)$ and $(i,i'')$ are edges of the original graph $G$, then the graph $H$ is also a subgraph of the original graph $G$, a contradiction. If $\{i',i\}\in E(G)$ but $\{i,i''\}\notin E(G)$, then $((P-\{i,i''\})\cup (ij'ji''))\cup (Q-(jj'j''))\cup (H-(P\cup Q))\in \mathcal{L}_{n,k}$ is in $G$, a contradiction. Similarly, if $\{i,i''\}\in E(G)$ but $\{i',i\}\notin E(G)$, then $((P-\{i',i\})\cup (i'ji))\cup (Q-(jj'j''))\cup (H-(P\cup Q))\in \mathcal{L}_{n,k}$ is in $G$, a contradiction. Suppose that $i,j\in V(H)$ and $i$ is the origin or terminus of a path $P$ in $H$ but $j\in V(Q)$, where $P$ and $Q=j_0j_1\cdots j'jj''\cdots j_t$ are two connected components of $H$. Assume that $\{i,i'\}\in P$. If $\{i,i'\}\in E(G)$, then $H$ is also in $G$, a contradiction. If $\{i.i'\}\notin E(G)$, then $((P-\{i,i'\})\cup (i'jj''\cdots j_t))\cup (j_0j_1\cdots j'i)\cup (H-(P\cup Q))\in \mathcal{L}_{n,k}$ is in $G$, a contradiction.
If $i\in V(H)$ but $j\notin V(H)$, then we can find a graph $G'\in \mathcal{L}_{n,k}$ in the original graph $G$ based on the similar analysis to Subcase \ref{subcase2.1.4}, a contradiction.

Combining all the cases, we conclude that $S_{ij}(G)$ is $\mathcal{L}_{n,k}$-free if $G$ is $\mathcal{L}_{n,k}$-free.
\end{proof}
In \cite{wang}, Wang proved that the shifting operation cannot reduce the number of the copies of $K_s$ and $K^*_{s,t}$.
\begin{lemma}[\cite{wang}]\label{ch2:lem2}
Let $G$ be a graph on vertex set $[n]$. For $1\le i<j\le n$, we have that
$$
\mathcal{N}(S_{ij}(G),K_s)\ge \mathcal{N}(G,K_s) \mbox{ and } \mathcal{N}(S_{ij}(G),K^*_{s,t})\ge \mathcal{N}(G,K^*_{s,t}).
$$
\end{lemma}
Suppose that $G$ is a graph on vertex set $[n]$. If $S_{ij}(G)=G$ holds for all $i,j$ with $1\le i<j\le n$, then we call $G$ a \textit{shifted graph}.
Further, if $G$ is a shifted graph, then for any $\{u,v\}\in E(G)$, $u'<u$ and $u'\neq v$, we always have $\{u',v\}\in E(G)$. otherwise, we have
$S_{u'u}(G)\neq G$, a contradiction.

Let $G$ be a graph on $n$ vertices, $\mathcal{P}$ a property defined on $G$, and $k$ a positive integer. We call the property $\mathcal{P}$  \textit{$k$-stable} if whenever $G+uv$ has the property
$\mathcal{P}$ and $d_{G}(u)+d_{G}(v)\ge k$, then $G$ itself has the property $\mathcal{P}$. In \cite{N-W}, Ning and Wang proved the property ``$\mathcal{L}_{n,k}$-free'' is $k$-stable.
\begin{lemma}[\cite{N-W}]\label{ch2:lem3}
Let $G$ be a graph on $n$ vertices. Suppose that $u,v\in V(G)$ with $d(u)+d(v)\ge k$. Then $G$ is $\mathcal{L}_{n,k}$-free if and only if $G+uv$ is $\mathcal{L}_{n,k}$-free.
\end{lemma}

\section{The generalized Tur\'{a}n number of spanning linear forests}\label{se3}
\quad\quad In this section, we determine the exact values of $ex(n,K_s,\mathcal{L}_{n,k})$ and $ex(n,K^*_{s,t},\mathcal{L}_{n,k})$ by characterizing all the shifted graph which contains a largest linear forest containing $k-1$ edges.

For $\lceil\frac{k+1}{2}\rceil\le m\le k$, we define a graph $H(n,k,m)$ on vertex set $[n]$ as follows. Let $A=[m]$, $B=[n]\setminus A$
and $C=[k-m]\subset A$. The edge set of $H(n,k,m)$ consists of all edges
between $B$ and $C$ together with all edges in $A$. In the following lemma, we characterize all the shifted graph
which contains a largest linear forest containing $k-1$ edges.

\begin{lemma}\label{ch3:lem1}
Let $G$ be a shifted graph on vertex set $[n]$ which contains a largest linear forest containing $k-1$ edges. Then $G$ is a subgraph of $H(n,k,m)$ for some $\lceil \frac{k+1}{2}\rceil\le m\le k$.
\end{lemma}
\begin{proof}
Let $G'$ be an $\mathcal{L}_{n,k}$-free graph on vertex set $[n]$ with maximum number of edges that containing $G$ as a subgraph. Then apply the shifting operation $S_{ij}$ to $G'$ for all $i,j$ with $1\le i<j\le n$. Finally, we obtain a graph $\tilde{G}$. Since $E(G)\subseteq E(G')$ and $G$ is a shifted graph,
then $G$ is also a subgraph of $\tilde{G}$. By Lemma \ref{ch2:lem1}, $\tilde{G}$ is also $\mathcal{L}_{n,k}$-free.

\begin{claim}
Vertex subset $\lceil\frac{k+1}{2}\rceil$ forms a clique in $\tilde{G}$.
\end{claim}
\begin{proof}
Suppose to the contrary, there exist $u_1,u_2$ in $V(\tilde{G})$ with $1\le u_1<u_2\le \lceil\frac{k+1}{2}\rceil$ such that $\{u_1,u_2\}\notin E(\tilde{G})$.
Since $G$ is a subgraph of $\tilde{G}$, $\tilde{G}$ contains a largest linear forest containing $k-1$ edges. It follows that there exists an edge $\{v_1,v_2\}\in E(\tilde{G})$ such that $v_2>v_1\ge \lceil\frac{k-1}{2}\rceil$. Then by $u_1\le \lceil\frac{k-1}{2}\rceil\le v_1$ and $\{v_1,v_2\}\in E(\tilde{G})$, we have $\{u_1,v_2\}\in E(\tilde{G})$. By $u_2\le \lceil\frac{k+1}{2}\rceil\le v_2$, we have $\{u_1,u_2\}\in E(\tilde{G})$, a contradiction.
Thus, $\{1,2,3,\cdots, \lceil\frac{k+1}{2}\rceil\}$ forms a clique in $\tilde{G}$.
\end{proof}
Let $m$ be the maximum integer such that $[m]$ forms a clique in $\tilde{G}$. Let $A=[m]$ and $B=[n]\setminus A$. If $m\ge k+1$, then we can obtain a path of length $k$, which is a graph in $\mathcal{L}_{n,k}$, a contradiction. Thus, we have $\lceil\frac{k+1}{2}\rceil\le m\le k$.
\begin{claim}
$B$ forms an independent set in $\tilde{G}$.
\end{claim}
\begin{proof}
Suppose to the contrary, there exist $u_1,u_2$ in $V(\tilde{G})$ with $m+1\le u_1<u_2\le n$ such that $\{u_1,u_2\}\in E(\tilde{G})$. Then for any $u\in A$,
since $u\le m<u_1$ and $m+1<u_2$, then $\{u,m+1\}$ is an edge of $\tilde{G}$. It follows that $[m+1]$ forms a clique of $\tilde{G}$, a contradiction.
Thus, the claim holds.
\end{proof}
\begin{claim}
For any vertex $v\in B$, we have $d_{\tilde{G}}(v)\le k-m$.
\end{claim}
\begin{proof}
Suppose to the contrary that $d_{\tilde{G}}(v)\ge k-m+1$ for some $y\in B$. Since $A$ is a maximum clique and $v\notin A$, it follows that there exists some
$u\in A$, such that $\{u,v\}\notin E(\tilde{G})$. Since $\tilde{G}$ is the one with maximum number of edges, we know $\tilde{G}+\{u,v\}$ is not $\mathcal{L}_{n,k}$-free. Since $d_{\tilde{G}}(u)\ge m-1$ and $d_{\tilde{G}}(v)\ge k-m+1$, then $d_{\tilde{G}}(u)+d_{\tilde{G}}(v)\ge k$. By Lemma \ref{ch2:lem3}, we have that $\tilde{G}$ is not $\mathcal{L}_{n,k}$-free, a contradiction.
\end{proof}

Combining all the claims, we conclude that $G$ is a subgraph of $H(n,k,m)$.
\end{proof}

\renewcommand{\proofname}{\bf Proof of Theorem \ref{ch1:th5}.}
\begin{proof}
When $k$ is odd, $K_k$ and $K_{\frac{k-1}{2}}\vee E_{n-\frac{k-1}{2}}$ are $\mathcal{L}_{n,k}$-free graphs with the required
number of $s$-cliques. The number of $s$-cliques is
$$
\begin{aligned}
&max\left\{\binom{k}{s}, \binom{\frac{k-1}{2}}{s}+\bigg(n-\frac{k-1}{2}\bigg)\binom{\frac{k-1}{2}}{s-1}\right\}\\
&=max\left\{\binom{k}{s}, \binom{\frac{k+1}{2}}{s}+\bigg(n-\frac{k+1}{2}\bigg)\binom{\frac{k-1}{2}}{s-1}\right\}.
\end{aligned}
$$
When $k$ is even, $K_k$ and $K_{\frac{k}{2}-1}\vee (E_{n-\frac{k}{2}-1}\cup K_2)$ are $\mathcal{L}_{n,k}$-free graphs with the required
number of $s$-cliques. The number of $s$-cliques is
$$
\begin{aligned}
&max\left\{\binom{k}{s}, \binom{\frac{k}{2}-1}{s}+\bigg(n-\frac{k}{2}+1\bigg)\binom{\frac{k}{2}-1}{s-1}+\binom{\frac{k}{2}-1}{s-2}\right\}\\
&=max\left\{\binom{k}{s}, \binom{\frac{k}{2}+1}{s}+\bigg(n-\frac{k}{2}-1\bigg)\binom{\frac{k}{2}-1}{s-1}\right\}.
\end{aligned}
$$
Therefore we only need to prove the upper bound.
Let $G$ be an $\mathcal{L}_{n,k}$-free graph on vertex set $[n]$ with the maximum number of $s$-cliques. Since adding edges cannot reduce the number of $s$-cliques, we assume
that $G$ is the one with maximum number of edges which is $\mathcal{L}_{n,k}$-free and $\mathcal{N}(G,K_s)$ is maximum. Clearly, we have that $G$ contains a largest linear forest consists of $k-1$ edges.
Otherwise, by adding one edge to $G$, we get a new graph $G'$ with more edges and $G'$ is also $\mathcal{L}_{n,k}$-free, a contradiction. By Lemmas \ref{ch2:lem1} and \ref{ch2:lem2}, we can further assume
$G$ is a shifted graph. Then by Lemma \ref{ch3:lem1}, we obtain that $G$ is a subgraph of $H(n,k,m)$ for some $\lceil\frac{k+1}{2}\rceil\le m\le k$.

If $s>k$, since $G$ is $\mathcal{L}_{n,k}$-free, it follows that $\mathcal{N}(G,K_s)=0$.

If $\lfloor\frac{k-1}{2}\rfloor+2\le s\le k$, we have
$$
\mathcal{N}(G,K_s)\le \mathcal{N}(H(n,k,m),K_s)=\binom{m}{s}\le \binom{k}{s}.
$$

If $2\le s\le \lfloor\frac{k-1}{2}\rfloor+1$, then
$$
\mathcal{N}(G,K_s)\le \mathcal{N}(H(n,k,m),K_s)=\binom{m}{s}+(n-m)\dbinom{k-m}{s-1}.
$$

Let
$$f(m)=\binom{m}{s}+(n-m)\dbinom{k-m}{s-1}.$$

By considering the second derivative, it is easy to check that $f(m)$ is a convex function despite of the parity of $k$.
Since $\lceil\frac{k+1}{2}\rceil\le m\le k$, it follows that
$$
\begin{aligned}
\mathcal{N}(G,K_s)&\le \mathcal{N}(H(n,k,m),K_s)\\
&\le max\left\{f(k),f\bigg(\left\lceil\frac{k+1}{2}\right\rceil\bigg)\right\}\\
&=max\left\{\binom{k}{s}, \binom{\lceil\frac{k+1}{2}\rceil}{s}+\bigg(n-\left\lceil\frac{k+1}{2}\right\rceil\bigg)\binom{\lfloor\frac{k-1}{2}\rfloor}{s-1}\right\}.
\end{aligned}
$$
Combining all the cases, we obtain that for $s\ge 2$ and $n\ge k+1$,
\begin{equation*}
ex(n,K_s,\mathcal{L}_{n,k})\le \mathcal{N}(G,K_s)\le max\left\{\binom{k}{s}, \binom{\lceil\frac{k+1}{2}\rceil}{s}+\bigg(n-\left\lceil\frac{k+1}{2}\right\rceil\bigg)\binom{\lfloor\frac{k-1}{2}\rfloor}{s-1}\right\}.
\end{equation*}
This complete our proof.
\end{proof}

\renewcommand{\proofname}{\bf Proof of Theorem \ref{ch1:th6}.}
\begin{proof}
When $k$ is odd, $K_k$ and $K_{\frac{k-1}{2}}\vee E_{n-\frac{k-1}{2}}$ are $\mathcal{L}_{n,k}$-free graphs with the required
number of $s$-cliques. The number of $s$-cliques is

\begin{align}
\nonumber max&\left\{\binom{k}{s+t}\binom{s+t}{t}, \binom{\frac{k-1}{2}}{s}\binom{n-s}{t}+\bigg(n-\frac{k-1}{2}\bigg)\binom{\frac{k-1}{2}}{s+t-1}\binom{s+t-1}{t}\right\}\\
\nonumber &=max\left\{{\binom{k}{s+t}\binom{s+t}{t}, \binom{\frac{k-1}{2}}{s}\binom{n-s}{t}+\bigg(n-\frac{k+1}{2}\bigg)\binom{\frac{k-1}{2}}{s-1}\binom{\frac{k-1}{2}-s+1}{t}}\right.\\
\nonumber &\left.{+\bigg(\binom{\frac{k+1}{2}}{s}-\binom{\frac{k-1}{2}}{s}\bigg)\binom{\frac{k+1}{2}-s}{t}}\right\}.
\end{align}
When $k$ is even, $K_k$ and $K_{\frac{k}{2}-1}\vee (E_{n-\frac{k}{2}-1}\cup K_2)$ are $\mathcal{L}_{n,k}$-free graphs with the required
number of $s$-cliques. The number of $s$-cliques is
\begin{align}
\nonumber max&\left\{{\binom{k}{s+t}\binom{s+t}{t},\binom{\frac{k}{2}-1}{s}\binom{n-s}{t}+\bigg(n-\frac{k}{2}+1\bigg)\binom{\frac{k}{2}-1}{s+t-1}\binom{s+t-1}{t}}\right.\\
\nonumber &\left.{+\binom{\frac{k}{2}-1}{s}\binom{\frac{k}{2}-1-s}{t-2}}+\binom{\frac{k}{2}-1}{s-1}\binom{\frac{k}{2}-s}{t-1}+\binom{\frac{k}{2}-1}{s-2}\binom{\frac{k}{2}+1-s}{t}\right\}\\
\nonumber &=max\left\{{\binom{k}{s+t}\binom{s+t}{t}, \binom{\frac{k}{2}-1}{s}\binom{n-s}{t}+\bigg(n-\frac{k}{2}-1\bigg)\binom{\frac{k}{2}-1}{s-1}\binom{\frac{k}{2}-1-s+1}{t}}\right.\\
\nonumber &\left.{+\bigg(\binom{\frac{k}{2}+1}{s}-\binom{\frac{k}{2}-1}{s}\bigg)\binom{\frac{k}{2}+1-s}{t}}\right\}.
\end{align}
Therefore we only need to prove the upper bound.
Let $G$ be an $\mathcal{L}_{n,k}$-free graph on vertex set $[n]$ with the maximum number of copies of $K^*_{s,t}$. Since adding edges cannot reduce the number of copies of $K^*_{s,t}$,
we assume $G$ is the one with maximum number of edges which is $\mathcal{L}_{n,k}$-free and $\mathcal{N}(G,K^*_{s,t})$ is maximum. By Lemmas \ref{ch2:lem1} and \ref{ch2:lem2}, we can further assume $G$ is a shifted graph. Then by Lemma \ref{ch3:lem1}, we obtain that $G$ is a subgraph of $H(n,k,m)$ for some $\lceil\frac{k+1}{2}\rceil\le m\le k$.

Let $\Psi_m(K^*_{s,t})$ be the set of all $K^*_{s,t}$ in $H(n,k,m)$, $i.e.$,
\begin{align}
\nonumber\Psi_m(K^*_{s,t})=&\left\{{(W_1,W_2):|W_1|=s,|W_2|=t \mbox{ and } (W_1,W_2)}\mbox{ forms }\right.\\
\nonumber &\left.{\mbox{ a copy of } K^*_{s,t} \mbox{ in } H(n,k,m)}\right\}.
\end{align}
Let $U=[m]$, $U_0=[k-m]$ and $U'=[n]\setminus U$. Now we enumerate the copies of $K^*_{s,t}$ in $H(n,k,m)$ by classifying $\Psi_m(K^*_{s,t})$ into three classes as follows:
$$
\begin{cases}
\Psi_1=\left\{(W_1,W_2)\in \Psi_m(K^*_{s,t}):W_1\subset U_0\right\};\\
\Psi_2=\left\{(W_1,W_2)\in \Psi_m(K^*_{s,t}):W_1\cap U'\neq \phi\right\};\\
\Psi_3=\left\{(W_1,W_2)\in \Psi_m(K^*_{s,t}):W_1\cap (U\setminus U_0)\neq \phi\right\}.
\end{cases}
$$
For the first class, since there are $\binom{k-m}{s}$ ways to choose $W_1$ and $\binom{n-s}{t}$ ways to choose $W_2$, it follows that
$$
|\Psi_1|=f_1(m)=\binom{k-m}{s}\binom{n-s}{t}.
$$
For the second class, since $U'$ is an independent set, there is exactly one vertex in $U'$ belonging to $W_1$ and all the other vertices in $W_1\cup W_2$ are contained in $U_0$.
It follows that
$$
|\Psi_2|=f_2(m)=(n-m)\binom{k-m}{s-1}\binom{k-m-s+1}{t}.
$$
For the third class, there are $\binom{m}{s}-\binom{k-m}{s}$ choices for $W_1$ and $\binom{m-s}{t}$ choices for $W_2$.
Thus, we have
$$
|\Psi_3|=f_3(m)=\bigg(\binom{m}{s}-\binom{k-m}{s}\bigg)\binom{m-s}{t}.
$$
By considering the second derivative, we have that $f_1(m)$, $f_2(m)$ and $f_3(m)$ are all convex function in $m$.

Let $f(m)=f_1(m)+f_2(m)+f_3(m).$
Then $\mathcal{N}(H(n,k,m),K^*_{s,t})=f(m)$ and $f(m)$ is a convex function in $m$. Thus, we have
\begin{align}
\nonumber \mathcal{N}(G,K^*_{s,t})&\le \mathcal{N}(H(n,k,m),K^*_{s,t})\\
\nonumber &\le max\left\{f(k),f\bigg(\left\lceil\frac{k+1}{2}\right\rceil\bigg)\right\}\\
\nonumber &\le max\left\{{\dbinom{k}{s}\dbinom{k-s}{t}, \dbinom{\left\lfloor\frac{k-1}{2}\right\rfloor}{s}\dbinom{n-s}{t}+}\right.\\
\nonumber &\left.{\bigg(n-\left\lceil\frac{k+1}{2}\right\rceil\bigg)\dbinom{\left\lfloor\frac{k-1}{2}\right\rfloor}{s-1}\dbinom{\left\lfloor\frac{k-1}{2}\right\rfloor-s+1}{t}+}\right.\\
\nonumber &\left.{\bigg(\dbinom{\left\lceil\frac{k+1}{2}\right\rceil}{s}-\dbinom{\left\lfloor\frac{k-1}{2}\right\rfloor}{s}\bigg)\dbinom{\left\lceil\frac{k+1}{2}\right\rceil-s}{t}}\right\}.
\end{align}
Thus, we complete the proof.
\end{proof}
\section{The generalized Tur\'{a}n number of spanning linear forests on bipartite case}\label{se4}
\quad\quad In this section, we determine the exact value of $ex_{bip}(n,K_{s,t},\mathcal{L}_{n,k})$. Consider the bipartite graph which is $\mathcal{L}_{n,k}$-free, we prove the following lemma.
\begin{lemma}\label{ch4:lem1}
Let $G$ be a bipartite graph with each part of equal size $n$. If $G$ is the one with maximum number of edges which is $\mathcal{L}_{n,k}$-free, then there
exists a subset $T$ of the vertices with $|T|=\lceil\frac{k-1}{2}\rceil$, such that all edges of $G$ are incident to at least one vertex of $T$.
\end{lemma}

The following version of the K\"{o}nig-Hall Theorem will be used in our proof.
\begin{theorem}[\cite{Lovasz}]\label{ch4:th1}
Let $G$ be a bipartite graph with $\nu(G)=k$. Then there exists a subset $T$ of the vertices with $|T|=k$, such that all edges of $G$ are incident to at least one vertex of $T$.
\end{theorem}
In the following, we give the proof of Lemma \ref{ch4:lem1} mainly by Theorem \ref{ch4:th1}.
\renewcommand{\proofname}{\bf Proof of Lemma \ref{ch4:lem1}}
\begin{proof}
By Theorem \ref{ch4:th1}, we only need to show that $\nu(G)=\lceil\frac{k-1}{2}\rceil$.
Since $G$ is the one with maximum number of edges which is $\mathcal{L}_{n,k}$-free, it follows that $G$ is not $\mathcal{L}_{n,k-1}$-free. Otherwise, by adding one edge to $G$, we get a new graph $G'$
with more edges which is $\mathcal{L}_{n,k}$-free, a contradiction.

When $k$ is odd, we can find a graph $H=K_{\frac{k-1}{2},n}\cup E_{n-\frac{k-1}{2}}$ which is $\mathcal{L}_{n,k}$-free and $\nu(H)=\frac{k-1}{2}$.
It follows that $\nu(G)\ge \frac{k-1}{2}$ and $e(G)\ge \frac{k-1}{2} n$ for an odd $k$. When $k$ is even, let $H_1=K_{\frac{k}{2}-1,n}$, $H_2=K_{1,n-\frac{k}{2}+1}$
and $H'$ is the union of $H_1$ and $H_2$ which the centre vertex of $H_2$ is same as some vertex from the vertex partite set of size $n$ in $H_1$. It is easy to know that the graph $H'$ is
$\mathcal{L}_{n,k}$-free and $\nu(H')=\frac{k}{2}$. It follows that $\nu(G)\ge \frac{k}{2}$ and $e(G)\ge \frac{k}{2} n-\frac{k}{2}+1$ for an even $k$.
Thus, we have $\nu(G)\ge \lceil\frac{k-1}{2}\rceil$.
We give the claim as follows.
\begin{claim}
$\nu(G)\le \lceil\frac{k-1}{2}\rceil$.
\end{claim}
\renewcommand{\proofname}{Proof.}
\begin{proof}
Suppose to the contrary, we have $\nu(G)\ge \lceil\frac{k+1}{2}\rceil$. Let $G=G[U,V]$ be a bipartite graph, where $|U|=|V|=n$. Let $U=\{u_1,u_2,\cdots,u_n\}$ and $V=\{v_1,v_2,\cdots,v_n\}$.
Assume $\nu(G)=x$ and $M_x=\bigcup_{i=1}^x u_iv_i$ is a maximum matching of $G$. Note that $\lceil\frac{k+1}{2}\rceil\le x\le k-1$.
It follows from $M_x$ is a maximum matching of $G$ that there exists no edge in $G[V(G)\setminus V(M_x)]$.
It follows from $G$ is a bipartite graph that $u_i$ and $v_i$ cannot be adjacent to same vertex in $V(G)\setminus V(M_x)$, where $u_iv_i\in M_x$.
For some edge $u_iv_i\in M_x$, $u_i$ and $v_i$ cannot be adjacent to distinct vertices in $V(G)\setminus V(M_x)$ at the same time.
Otherwise, if the vertex $u_{i'}$ is adjacent to $v_{i''}$ and $v_{i'}$ is adjacent to $u_{i''}$ for $u_iv_i\in M_x$ and $u_{i''}, v_{i''}\in V(G)\setminus V(M_x)$, then
by substituting the edge $u_{i'}v_{i'}$ with the two edges $u_{i'}v_{i''}$ and $u_{i''}v_{i'}$, we can find a copy of matching with more edges than $M_x$, a contradiction.
We denote the subset of vertices of $V(M_x)$ whose each vertex is adjacent to at least one vertex in $V(G)\setminus V(M_x)$ by $U$.
Since $G$ is $\mathcal{L}_{n,k}$-free, it follows that $|U|< \lceil\frac{k-1}{2}\rceil$. Otherwise, we can find a copy of $\mathcal{L}_{n,k}$, a contradiction.
Thus, there exist at most $|U|(n-x)$ edges between the vertex set $V(M_x)$ and vertex set $V\setminus M_x$.
Note that the complete bipartite graph $K_{\lceil\frac{k+1}{2}\rceil,\lceil\frac{k-1}{2}\rceil}$ contains a path $P_{k+1}$. Therefore there exist no such subgraph in $G$.
It follows that $G[S_1,S_2]$ cannot be a copy of $K_{\lceil\frac{k+1}{2}\rceil,\lceil\frac{k-1}{2}\rceil}$, where $S_1\subset U, S_2\subset V$, $|S_1|=\lceil\frac{k+1}{2}\rceil$ and
$|S_2|=\lceil\frac{k-1}{2}\rceil$. In order to make $G[S_1,S_2]$ do not contain a path $P_{k+1}$, we need to delete at least $\lceil\frac{k-1}{2}\rceil$ edges in $G[S_1,S_2]$.
There are $2\binom{x}{\lceil\frac{k+1}{2}\rceil}\binom{x}{\lceil\frac{k-1}{2}\rceil}$ copies of such subgraph. It follows that there are at most $x^2-2\binom{x}{\lceil\frac{k+1}{2}\rceil}\binom{x}{\lceil\frac{k-1}{2}\rceil}\lceil\frac{k-1}{2}\rceil$ edges in $G[M_x]$.
Thus, the number of edges of $G$ suffices the following condition.
\begin{align}
\nonumber e(G)&< |U|(n-x)+x^2-2\binom{x}{\lceil\frac{k+1}{2}\rceil}\binom{x}{\lceil\frac{k-1}{2}\rceil}\left\lceil\frac{k-1}{2}\right\rceil\\
\nonumber &\le \bigg(\left\lceil\frac{k-1}{2}\right\rceil-1\bigg)(n-x)+x^2-2\binom{x}{\lceil\frac{k+1}{2}\rceil}\binom{x}{\lceil\frac{k-1}{2}\rceil}\left\lceil\frac{k-1}{2}\right\rceil\\
\nonumber &\le \bigg(\left\lceil\frac{k-1}{2}\right\rceil-1\bigg)(n-x)+x^2-2\left\lceil\frac{k-1}{2}\right\rceil\left\lceil\frac{k+1}{2}\right\rceil\\
\nonumber &<\bigg(\left\lceil\frac{k-1}{2}\right\rceil-1\bigg) n,
\end{align}
which contradicts with the facts $e(G)\ge \lceil\frac{k-1}{2}\rceil n$ for an odd $k$ and $e(G)\ge \lceil\frac{k-1}{2}\rceil n-\frac{k}{2}+1$ for an even $k$.
\end{proof}
It follows that $\nu(G)=\frac{k}{2}$. Thus, the lemma holds.
\end{proof}

\renewcommand{\proofname}{\bf Proof of Theorem \ref{ch1:th7}}
\begin{proof}
When $k$ is odd, the bipartite graph $K_{\frac{k-1}{2},n}\cup E_{n-\frac{k-1}{2}}$ is the one with the required number of $K_{s,t}$.
When $k$ is even, let $H'$ be the union of $H_1=K_{\frac{k}{2}-1,n}$ and $H_2=K_{1,n-\frac{k}{2}+1}$ which the centre vertex of $H_2$ is same as some vertex from the vertex partite set of size $n$ in $H_1$.
Then $H'$ is the one with the required number of $K_{s,t}$.

Thus, we only need to prove the upper bound.
Let $G$ be a bipartite graph which contain two vertex sets $X$ and $Y$ with equal size $n$. Suppose that $G$ is the one with the maximum number of copies of $K_{s,t}$ which is $\mathcal{L}_{n,k}$-free.
We further assume that $G$ is the one with maximum number of edges which is
$\mathcal{L}_{n,k}$-free and $\mathcal{N}(G,K_{s,t})$ is maximum. Then by Lemma \ref{ch4:lem1}, there exists a subset $T\subset X\cup Y$ with $|T|=\lceil\frac{k-1}{2}\rceil$, such that all edges of $G$ are adjacent to
at least one vertex of $T$.

Let $X_1=X\cap T$, $Y_1=Y\cap T$, $X_2=X\setminus T$ and $Y_2=Y\setminus T$. Define $G^*$ to be a bipartite graph on vertex sets $X$ and $Y$ so that $G^*[X_1,Y]$ and $G^*[X,Y_1]$ are complete bipartite graphs
and $G^*[X_2,Y_2]$ is an empty graph. It is easy to see that $G$ ia a subgraph of $G^*$. It follows that $\mathcal{N}(G,K_{s,t})\le \mathcal{N}(G^*,K_{s,t})$.

Let $(S,T)$ be an ordered pair such that $S\subset X$ and $T\subset Y$ with $|S|=s$ and $|T|=t$, and $G^*[S,T]$ is a complete bipartite graph. Clearly, each copy of $K_{s,t}$ in $G^*$ is identified by such an ordered pair.
Since $G^*[X_2,Y_2]$ is an empty graph, it follows that at least one of $S\subset X_1$ and $T\subset Y_1$ holds. Let $|X_1|=x$. Since $|X_1|+|Y_1|=|T|=\lceil\frac{k-1}{2}\rceil$, it follows that $|Y_1|=\lceil\frac{k-1}{2}\rceil-x$.
Thus, we have
$$
\mathcal{N}(G^*,K_{s,t})=\dbinom{x}{s}\dbinom{n}{t}+\dbinom{n}{s}\dbinom{\lceil\frac{k-1}{2}\rceil-x}{t}-\dbinom{x}{s}\dbinom{\lceil\frac{k-1}{2}\rceil-x}{t}.
$$
Let
$$
f_{s,t}(x)=\dbinom{x}{s}\dbinom{n}{t}+\dbinom{n}{s}\dbinom{\lceil\frac{k-1}{2}\rceil-x}{t}-\dbinom{x}{s}\dbinom{\lceil\frac{k-1}{2}\rceil-x}{t}.
$$
By considering the second derivative, it is easy to check that $f_{s,t}(x)$ is a convex function.
\begin{case}
$k$ is odd.
\end{case}
For $s=t$, we have
$$
\mathcal{N}(G,K_{s,s})\le \mathcal{N}(G^*,K_{s,s})\le max\left\{f_{s,s}(0),f_{s,s}\bigg(\frac{k-1}{2}\bigg)\right\}=\dbinom{\frac{k-1}{2}}{s}\dbinom{n}{s}.
$$
For $s\neq t$, since $f_{s,t}(x)+f_{t,s}(x)$ is also a convex function, we have
\begin{align}
\nonumber\mathcal{N}(G,K_{s,t})&\le \mathcal{N}(G^*,K_{s,t})\\
\nonumber &\le max\left\{f_{s,t}(0)+f_{t,s}(0),f_{s,t}\bigg(\frac{k-1}{2}\bigg)+f_{t,s}\bigg(\frac{k-1}{2}\bigg)\right\}\\
\nonumber &=\dbinom{\frac{k-1}{2}}{s}\dbinom{n}{t}+\dbinom{\frac{k-1}{2}}{t}\dbinom{n}{s}.
\end{align}
\begin{case}
$k$ is even.
\end{case}
In this case, it is easy to see that $f_{s,t}(x)$ attains its maximum value at $x=1$ or $x=\frac{k}{2}-1$. Otherwise, we can find a copy of $P_{k+1}$ in $G^*$, a contradiction.
For $s=t$, we have
\begin{align}
\nonumber\mathcal{N}(G,K_{s,s})&\le \mathcal{N}(G^*,K_{s,s})\le max\left\{f_{s,s}(1),f_{s,s}\bigg(\frac{k}{2}-1\bigg)\right\}\\
\nonumber &=\dbinom{1}{s}\dbinom{n}{s}+\dbinom{\frac{k}{2}-1}{s}\dbinom{n}{s}-\dbinom{1}{s}\dbinom{\frac{k}{2}-1}{s}
\end{align}
If $s=1$, then
$$
\mathcal{N}(G,K_{s,s})\le \mathcal{N}(G^*,K_{s,s})\le max\left\{f_{s,s}(1),f_{s,s}\bigg(\frac{k}{2}-1\bigg)\right\}=\frac{k}{2}n-\frac{k}{2}+1.
$$
If $s\ge 2$, then
$$
\mathcal{N}(G,K_{s,s})\le \mathcal{N}(G^*,K_{s,s})\le max\left\{f_{s,s}(1),f_{s,s}\bigg(\frac{k}{2}-1\bigg)\right\}=\dbinom{\frac{k}{2}-1}{s}\dbinom{n}{s}.
$$
For $s\neq t$, since $f_{s,t}(x)+f_{t,s}(x)$ is also a convex function, we have
\begin{align}
\nonumber\mathcal{N}(G,K_{s,t})&\le \mathcal{N}(G^*,K_{s,t})\\
\nonumber &\le max\left\{f_{s,t}(1)+f_{t,s}(1),f_{s,t}\bigg(\frac{k}{2}-1\bigg)+f_{t,s}\bigg(\frac{k}{2}-1\bigg)\right\}\\
\nonumber &=\bigg(\dbinom{\frac{k}{2}-1}{t}+\dbinom{1}{t}\bigg)\dbinom{n}{s}+\bigg(\dbinom{\frac{k}{2}-1}{s}+\dbinom{1}{s}\bigg)\dbinom{n}{t}\\
\nonumber &-\dbinom{\frac{k}{2}-1}{s}\dbinom{1}{t}-\dbinom{\frac{k}{2}-1}{t}\dbinom{1}{s}.
\end{align}
If $s=1$ and $t\ge 2$, then
$$
\mathcal{N}(G,K_{s,t})\le \mathcal{N}(G^*,K_{s,t})=\frac{k}{2}\dbinom{n}{t}+(n-1)\dbinom{\frac{k}{2}-1}{t}.
$$
If $s\ge 2$ and $t=1$, then
$$
\mathcal{N}(G,K_{s,t})\le \mathcal{N}(G^*,K_{s,t})=\frac{k}{2}\dbinom{n}{s}+(n-1)\dbinom{\frac{k}{2}-1}{s}.
$$
If $s,t\ge 2$, then
$$
\mathcal{N}(G,K_{s,t})\le \mathcal{N}(G^*,K_{s,t})=\dbinom{\frac{k}{2}-1}{t}\dbinom{n}{s}+\dbinom{\frac{k}{2}-1}{s}\dbinom{n}{t}.
$$
Thus, the theorem holds.
\end{proof}

\end{document}